\RequirePackage{fix-cm}
\documentclass[smallextended]{svjour3}       %
\smartqed  %
\usepackage{graphicx}
\usepackage{dsfont}
\usepackage{amsmath,amssymb}
\usepackage{booktabs}

\usepackage[colorlinks]{hyperref}
\usepackage[shortcuts]{extdash}
\usepackage{enumerate}

\newtheorem{csp}{CSP}

\newcommand{\ra}{\rightarrow}
\newcommand{\N}{\mathds N}
\newcommand{\ms}{\mapsto}
\newcommand{\ol}{\overline}
\newcommand{\mc}{\mathcal}
\newcommand{\id}{\operatorname{id}}

\begin{document}

\title{The semigroups of order 9 and their automorphism groups}

\author{Andreas Distler        \and Tom Kelsey}

\institute{Andreas Distler \at
              Centro de \'{A}lgebra da Universidade de Lisboa, 1649-003 Lisboa, Portugal\\
              Tel.: +351 21 790 4716\\
              Fax: +351 21 795 4288\\
              \email{adistler@fc.ul.pt}           %
           \and
           Tom Kelsey \at
             School of Computer Science,
University of St Andrews, KY16 9SX, UK\\
}

\date{Received: date / Accepted: date}

\maketitle

\begin{abstract}
We report the number of semigroups with $9$ elements 
up to isomorphism or anti-isomorphism to be
$52\,989\,400\,714\,478$ and up to isomorphism to be
$105\,978\,177\,936\,292$. We obtained these results by combining
computer search with recently published formulae for the number of
nilpotent semigroups of degree 3. We further provide a complete
account of the automorphism groups of the semigroups with at most $9$
elements. We use this information to deduce that there are
$148\,195\,347\,518\,186$ distinct associative binary operations on an
$8$-element set and $38\,447\,365\,355\,811\,944\,462$ on a
$9$-element set.

\end{abstract}

\section{Introduction}
\label{intro}

Classification of finite semigroups of a given order goes
back to the 1950s when Tamura
undertook 
hand calculations, first for orders $2$ and $3$ \cite{Tam53} and later for
order $4$ \cite{Tam54}. 
Around the same time Forsythe
introduced computer search to the problem~\cite{For55} implementing a
backtrack algorithm to find semigroups on a
$4$ element set. Subsequently various authors refined his
approach~\cite{MS55,Ple67,JW77,SYT94}, so that by 1994 semigroups were
classified up to order 8. These semigroups are nowadays available in
the data library \textsf{Smallsemi}~\cite{smallsemi}.
A recent advance in the enumeration of semigroups are
the formulae derived in~\cite{DM12} for the numbers of semigroups $S$
for which $|\{abc\mid a,b,c\in S\}|=1$ and $S$ is not a zero semigroup.
Such semigroups are called \emph{nilpotent of degree 3}. These were
used in the attempt to establish an asymptotic lower
bound on the number of all semigroups on a finite
set~\cite{KRS76}. The analysis in~\cite{SYT94} shows that for order
$8$ about $99.4 \%$ of the semigroups are nilpotent of degree~$3$. 

Research investigating the automorphism groups of semigroups of a
given order has a far briefer history. It has long been known that
every group appears as the automorphism group of some semigroup which
is a consequence of the analogous result for graphs~\cite{Fru39}, but
the first algorithm to compute automorphism groups of semigroups was
only presented in~\cite{ABMN10} where Ara\'ujo {\it et al.}~compute as
one application of their general method the automorphism groups of
semigroups of orders at most $7$.

Naturally the principal aim of the aforementioned investigations
was to consider `structural types' of semigroups of the given order
rather than distinct semigroups on a set of that size. Two
semigroups $S$ and $R$ are \emph{anti-isomorphic} if one
is isomorphic to the dual of the other, that is if there
exists a bijection $\sigma: S \ra R$ such that $\sigma(ab) =
\sigma(b) \sigma(a)$ for all $a,b \in S$; in this case $\sigma$
is an \emph{anti-isomorphism}. For short we write
\emph{(anti\=/)isomorphic} to mean isomorphic or anti-isomorphic and
analogously write \emph{(anti\=/)isomorphism}. Classification of
semigroups has mainly been done up to (anti\=/)isomorphism. The
connection to a classification up to isomorphism is provided by those
semigroups which are anti-isomorphic to themselves, that is isomorphic
to their dual, called \emph{self-dual}.

In this paper we enumerate semigroups of order $9$ up to isomorphism
and (anti\=/)isomorphism. 
The number up to (anti\=/)isomorphism, $52\,989\,400\,714\,478$, was
first reported in~\cite{DK09},  without  
the explanation and justification that are provided here.
The number up to isomorphism is $105\,978\,177\,936\,292$.
We also classify the semigroups of orders $8$ and $9$ by their
automorphism groups (see Tables~\ref{aut_8} and~\ref{aut_9}) and
deduce that the number of distinct semigroups on a set with $8$ elements
is $148\,195\,347\,518\,186$ and on a set with $9$ elements is
$38\,447\,365\,355\,811\,944\,462$. We find that only a small proportion of
the subgroups of the symmetric groups of degrees $8$ and $9$ are
isomorphic to the automorphism group of any such semigroup
(Table~\ref{tab_autgrp_isos}) and prove in particular that the
automorphism group of a semigroup is transitive if and only if it is a
rectangular band (Propostion~\ref{prop_autgrp}). 

We obtain semigroups and their automorphism groups by computer
search. For the enumeration it suffices to count those semigroup that
are \emph{not} nilpotent of degree~$3$, since their numbers are
known~\cite{DM12}. We perform a series of computations to find
multiplication tables of the remaining semigroups, utilising an
approach similar to that described in~\cite{DK09}: we model the search
as a family of constraint satisfaction problems and execute the
constraint solver \textsf{Minion}~\cite{GJM06} to get the
solutions. The computer algebra system \textsf{GAP}~\cite{GAP4} is
used in the preparation of the input files, in particular for
calculations to avoid (anti\=/)isomorphic
solutions. We also use \textsf{Minion} to search for
automorphism groups and \textsf{GAP} to identify their isomorphism
types.

In the forthcoming section we explain how we represent semigroups and
isomorphisms respectively anti-isomorphism between them in an adequate
way for the computer search.
Section~\ref{sec_basic_model}
contains an introduction to Constraint Satisfaction, and a description of a 
formal model for finding canonical representatives of semigroups of a
given order up to (anti\=/)isomorphism. We also describe adaptations
of the model that allow us to find automorphism groups of semigroups,
and to find self-dual semigroups. 
In Section~\ref{sec_cases} we replace the initial model with an
equivalent family of models following an idea introduced by
Plemmons in the classification of semigroups of order
6~\cite{Ple67}.
Additionally we extend the idea proving a unified framework for this
approach. The extended approach allows one to incorporate mathematical
knowledge more easily into the search for particular types of
semigroups, a method that we apply in Section~\ref{sec_bands} to the
enumeration of bands.
In Section~\ref{sec_sg9} we report our computational experience and
give detailed classification results for semigroups of order $9$ by
various properties. For comparative purposes, we also show the
equivalent results for smaller orders in certain cases.
In the final section we describe our approach to the computation of
the automorphism groups of semigroups. We present results for
semigroups of order at most 9 up to isomorphism and up to
(anti\=/)isomorphism. 
We use this knowledge to calculate the numbers
of distinct semigroups on a set with at most $9$ elements, and discuss
which subgroups of the symmetric group are isomorphic to the
automorphism group of a semigroup.

\section{Preliminaries}
\label{sec:1}

In this paper the elements of a semigroup will mostly be
$\{1,2,\dots,n\}, n\in\N$ which we abbreviate by $[n]$. As usual the
multiplication table of a semigroup $S$ is the
square matrix $T_S =(t_{a,b})_{a,b\in S}$ with $t_{a,b}=ab$. If the
underlying set of $S$ is $[n]$ we may assume that the rows and columns
of the multiplication table are indexed according to their position in
the table which allows us to omit the row and column header. Under this
convention every square matrix of size $n$ with entries in $[n]$
uniquely defines a binary operation on $[n]$ and we denote the set of
all such matrices by $\Omega_n$.

We want to describe isomorphisms and anti-isomorphisms in terms of
multiplication tables. Consider first two isomorphic semigroups $S$
and $R$ on $[n]$ and a permutation $\pi$ in the symmetric group $S_n$
which is an isomorphism from $S$ to $R$. Given that $T_S=
(t_{i,j})_{i,j\in [n]}$ is the multiplication table of $S$ it follows
that $((t_{i,j})^{\pi})_{i^{\pi},j^{\pi}\in [n]}$ is the
multiplication table of $R$. And if $S$ and $R$ are anti-isomorphic and
$\pi$ an anti-isomorphism from $S$ to $R$ then
$((t_{j,i})^{\pi})_{i^{\pi},j^{\pi}\in [n]}$ is the multiplication
table of $R$. Hence we can capture isomorphism and anti-isomorphism in
the following action $\phi:\Omega_n\times (S_n\times C_2) \ra
\Omega_n$ sending a multiplication table $T\in\Omega_n$ to
\begin{equation}
\label{eq_action}
T^{(\pi,c)} = 
\begin{cases}
((t_{i,j})^{\pi})_{i^{\pi},j^{\pi}\in [n]} & \textrm{if} \; c = 1_{C_2}\\
((t_{j,i})^{\pi})_{i^{\pi},j^{\pi}\in [n]} & \textrm{otherwise}.
\end{cases}
\end{equation}
The orbits of this action are sets of those multiplication tables
which define (anti\=/)isomorphic binary operations. 

To avoid confusion and repetition of  similar arguments we will
throughout the paper use the action given in \eqref{eq_action} only,
incorporating both isomorphism and anti-isomorphism. Considerations up
to isomorphism using an action of $S_n$ on $\Omega_n$ are left to the
reader.
\section{Enumeration using Constraint Satisfaction}
\label{sec_basic_model}

As in previous classifications of semigroups we search their
multiplication tables. To formalise our task we use the language of
Constraint Satisfaction, a technique developed to model and solve
discrete combinatorial problems, and start by giving basic
definitions.

\begin{definition}
\label{def_csp}
A \emph{constraint satisfaction problem (CSP)} is a triple $(V,D,C)$,
consisting of a finite set $V$ of \emph{variables}, a finite set $D$,
called the \emph{domain}, of \emph{values}, and a set $C$ containing
subsets of $\{h\mid h:V\ra D\}$ called \emph{constraints}.
\end{definition}

In practice,  instead of being subsets of the set of all functions
from $V$ to $D$, constraints are formulated as conditions
defining such subsets. It then becomes intuitively clear that one is
looking for assignments of values in the domain of a CSP to all
variables such that no constraint is violated. This idea is formalised
in the next definition. 

\begin{definition}
Let $L=(V,D,C)$ be a CSP. A partial function $p: V \ra D$ is an
\emph{instantiation}. An instantiation $p$ \emph{satisfies} a
constraint if there exists a function $h$ in the constraint, such
that $h(x)=p(x)$ for all $x\in V$ on which $p$ is defined. An
instantiation is \emph{valid}, if it satisfies all the constraints in $C$.
An instantiation defined on all variables is \emph{total}. A valid,
total instantiation is a \emph{solution} to $L$. The number of
solutions of $L$ will be denoted by $\#L$.
\end{definition}

\subsection{Counting all semigroups}

We formulate a CSP  which has those multiplication tables
in $\Omega_n$ as solutions that define an associative multiplication:

\begin{csp}
\label{CSP_basic}
For $n\in \N$ define a CSP $L_n=(V_n,D_n,C_n)$. The set $V_n$ consists
of $n^2$ variables $\{ t_{i,j} \mid 1 \leq i,j \leq n \}$, one for
each position in an $(n \times n)$-multiplication table, having domain
$D_n=[n]$. The constraints in $C_n$ are
\begin{equation}
\label{const_assoc}
t_{t_{i,j},k} = t_{i,t_{j,k}} \textrm{ for all } i,j,k \in [n],
\end{equation}
reflecting associativity.
(Note that \eqref{const_assoc} is a slight abuse of notation: using
a variable as an index shall refer to its value.)
\end{csp}

The multiplication table defined by a solution of $L_n$ from
CSP~\ref{CSP_basic} will be associative due to the constraints $C_n$
and, in turn, the table of every associative multiplication fulfils
the constraints in $C_n$. Thus the valid, total
instantiations for $L_n$ correspond to the semigroups on
$[n]$. As the constraints $C_n$ enforcing associativity will be
present in every following model, the solutions will always define
semigroups and are often referred to as such.

The number of all different semigroups on $[n]$ grows rapidly with $n$
and most of the semigroups are nilpotent of degree
$3$~\cite{KRS76}. As the construction for nilpotent
semigroups of degree $3$ on $[n]$ is also given in~\cite{KRS76}, they do not have to be searched for. We
forbid these by requiring that not all multiplications of three
elements give the same result. Adding the constraint
\begin{equation}
\label{const_3nil}
\exists i,j,k,q,r,s \in [n]: t_{i,t_{j,k}} \neq t_{q,t_{r,s}}
\end{equation}
to $C_n$ yields the CSP, denoted as  $L_n^{-3}$, having as solutions all different semigroups on
$[n]$,  which are neither nilpotent of degree $3$ nor a
zero semigroup.

\subsection{Counting up to (anti\=/)isomorphism}
\label{ssec_SB}

Our primary aim is not to find all semigroups on $[n]$,
but rather to find representatives for all types of structurally different
semigroups, where structurally different means up to (anti\=/)isomorphism.

With increasing $n$ it becomes -- due to the large number of solutions
-- very quickly impractical to test for every two semigroups from
the solutions of $L_n$ or $L_n^{-3}$ whether they are 
(anti\=/)isomorphic.
Instead we shall define a canonical solution for
each class following a standard approach in the classification of
algebraic and combinatorial structures. To make the test for
canonicity an integral part of the CSP we adapt a common symmetry
breaking technique from Constraint Satisfaction.

We first need another way to describe a solution of a CSP.
A \emph{literal} (also called \emph{variable-value pair}) of a CSP
$L=(V,D,C)$ is an element in the Cartesian product $V\times
D$. Literals are denoted in the form $(x=k)$ with $x\in V$ and $k\in
D$. An instantiation $p$ corresponds to the set of literals
$\{(x=p(x))\mid p \mbox{ is defined on } x\}$, which uniquely determines
$p$ (but not every set of literals yields an instantiation). In
particular we get an action of $S_n\times C_2$ on literals, induced
from the action \eqref{eq_action} on multiplication tables. 

\begin{equation}
\label{lit_map}
(t_{i,j}=k)^{(\pi,c)} = 
\begin{cases}
(t_{i^{\pi},j^{\pi}} = k^{\pi}) & \textrm{if} \; c = 1_{C_2}\\
(t_{j^{\pi},i^{\pi} }= k^{\pi}) & \textrm{otherwise}.
\end{cases}
\end{equation}

Given a fixed ordering $(\chi_1,\chi_2,\dots,\chi_{|V||D|})$ of all
literals in $V\times D$, an instantiation $p$ can be represented as a
bit vector of length $|V||D|$. The bit in the $i$-th position is 1 if
$\chi_i$ is contained in the set of literals corresponding to $p$ and
otherwise the bit is 0. The resulting bit vector for the instantiation
$p$ will be denoted by $(\chi_1,\chi_2,\dots,\chi_{|V||D|})_{p}$.

Sets of solutions of $L_n$ and $L^{-3}_n$ that lead to (anti\=/)isomorphic
semigroups are orbits under the action of $S_n \times C_2$ given in
\eqref{eq_action}. There is one solution in each orbit for which the
corresponding bit vector is lexicographic maximal, which we take to be the
property identifying the canonical solution in the orbit. We denote
the standard lexicographic order on vectors by $\prec$.
We obtain a new CSP $\ol{L}_n$ from $L_n$, respectively
$\ol{L}_n^{-3}$ from $L^{-3}_n$, by adding, for all non-identity 
$g \in S_n \times C_2$, the constraint consisting of those functions
$h:V\ra D$ for which
\begin{equation}
\label{const_lex}
(\chi_1^{g}, \chi_2^{g}, \dots, \chi_{|V||D|}^{g})_{h}
\preceq (\chi_1,\chi_2,\dots,\chi_{|V||D|})_{h}.
\end{equation}
The solutions of the new CSPs 
are pairwise not (anti\=/)isomorphic, because $\ol{L}_n$
respectively $\ol{L}_n^{-3}$ have as solutions all canonical tables
from orbits of solutions of $L_n$ respectively $L^{-3}_n$. Hence
$\#\ol{L}_n$ equals the number of semigroups of order $n$ up
to (anti\=/)isomorphism, while $\#\ol{L}_n^{-3}$ equals the number of
semigroups of order $n$ that are not nilpotent of degree $3$ nor a
zero semigroup up to (anti\=/)isomorphism.

There is a computational drawback of the method explained in this
section to avoid (anti\=/)isomorphic solutions: the number of
canonicity constraints~\eqref{const_lex} to be added to $L_n$ to
obtain $\ol{L}_n$ is $2n! - 1$ and their length is $n^3$, which makes
the space requirements for the formulation of the constraints grow
very large already for small values of $n$. 
We can improve the situation to some extent by shortening
in \eqref{const_lex} the vectors on both sides depending on $g$ without
influencing the constraint. The technique we use is based on
\cite[Rule 1]{FH03} and the easiest example of its application is the
removal of literals which appear at the same position in both
vectors. It remains the more significant problem that the number of
constraints grows superexponentially with $n$. In Section~\ref{sec_cases}
we will explain an approach that 
ultimately overcomes this obstacle in our specific enumeration
problem.

\subsection{Finding automorphisms and self-dual semigroups} 
\label{ssec:autos}

A straightforward variation of the canonicity
constraints~\eqref{const_lex} allows to identify or prescribe
automorphisms of solutions of the CSP. A bijection $\pi\in S_n$ is an
automorphism if equality holds in constraint~\eqref{const_lex}
corresponding to $(\pi,1_{C_2})$. We can now either record for each
solution for which of the constraints equality holds; or alternatively
specify the automorphism group in advance, requiring equality
for the constraints corresponding to permutations in the
chosen group and strict inequality for all other permutations.

A similar approach can be used to identify self-dual semigroups. If we
require equality to hold for at least one of the
constraints~\eqref{const_lex} corresponding to an
anti-isomorphism (that is an element in
$S_n\times C_2$ with non-trivial $C_2$ component) then the solutions
will be exactly the self-dual semigroups of the original CSP. 

\section{Families of CSPs}
\label{sec_cases}

Possible enhancements of the CSP $L_n$ are restricted  by
the fact that not much can be said about the multiplication table of a
semigroup in general without knowing any of the entries. We adapt an
idea from~\cite{Ple67} to make the search for multiplication tables of
semigroups more efficient. Instead of running a single computation,
the search is split into cases depending on the diagonal of the
multiplication table. A major advantage of this approach is
that not all diagonals have to be considered when
searching for semigroups up to (anti\=/)isomorphism. 

In our adaptation we formulate one CSP for every diagonal.  
Note that the diagonals of multiplication tables
naturally correspond to functions from $[n]$ to itself. For a table
$T=(t_{i,j})_{i,j\in [n]}$ define a function $f_T:[n]\ra [n], i\ms
t_{i,i}$. Two tables can lead to the same function, but the
correspondence between diagonals and functions from $[n]$ to itself is
a bijection.

\begin{csp}
\label{csp_diag}
Given a function $f: [n] \rightarrow [n]$
define a CSP $L_f=(V_n,D_n,C_f)$ based on $L_n=(V_n,D_n,C_n)$ from
CSP~\ref{CSP_basic} by adding for all $i \in [n]$ the constraint
\begin{equation}
\label{const_fix_diag}
t_{i,i} = f(i)
\end{equation}
to $C_n$ to obtain $C_f$.
\end{csp}

The solutions to $L_f$ are all multiplication tables in $\Omega_n$
defining a semigroup in which the square of the element $i$ is given
by $f(i)$. In other words, the entries on the diagonal of the
multiplication table are specified {\it a priori}. We note that for
some functions $f$ the CSP $L_f$ will not have any solutions.
For example every finite semigroup has at
least one idempotent, which yields that every function without a fixed
point leads to a CSP without solutions.         

For a set $\mc{F}$ of functions from $[n]$ to $[n]$,
denote by $\mc{L}_{\mc{F}}$ the family of CSPs $\{ L_f \mid f \in
\mc{F}\}$. Let $\mc{F}_n$ denote the set of all functions with at
least one fixed point from $[n]$
to $[n]$. Then the CSPs in $\mc{L}_{\mc{F}_n}$ have together
the same solutions as $L_n$.
To select a smaller subset of functions in ${\mc{F}_n}$ such
that the corresponding instances still contain every type of semigroup
of order $n$
up to (anti\=/)isomorphism, we use the following lemma which gives
conditions in a general setting. We shall apply it again to a
different family of CSPs in Section~\ref{sec_bands}. Many more
applications can be found in~\cite[Chapter 5]{Dis10}.

\begin{lemma}
\label{lem_case_split}
Let $\mc{L} = \{ L_x \mid x \in X\}$ be a family of CSPs with
disjoint solution sets, and let $\mc{T}$
be a superset of all solutions.
Further let 
\[
\phi: \mc{T} \times G \rightarrow \mc{T}, (T,g) \mapsto T^g
\]
be an action of a group $G$ mapping solutions to solutions and
let $\psi: \mc{T} \rightarrow X$ be a surjective function.

If each solution $T$ of one of the CSPs in $\mc{L}$ is a solution of
$L_{\psi(T)}$, and if $\phi^{\psi}$ is an induced action of
$G$ on $X$ (that is, $x^g= \psi(T^g)$ for $x=\psi(T)$ is
well-defined), then the following statements hold.
\begin{enumerate}[(i)]
\item Let $Y \subseteq X$ contain at least one element of every
  orbit from $X$ under the induced action $\phi^{\psi}$. Then
  the solutions of $\{L_y \mid y \in Y\}$ contain at least one element
  from every orbit of solutions under the action of $\phi$.
\item Let $S \in L_x$ and $T \in L_y$. If $S$ is equivalent to $T$,
  then $x$ is equivalent to~$y$.
\item Let $T \in L_x$. Then the set of solutions of $L_x$ equivalent
  to $T$ equals the orbit of $T$ under the stabiliser of $x$ in $G$.
\end{enumerate}
\end{lemma}
\begin{proof}
(i): Let $T$ be a solution of one of the CSPs in $\mc{L}$. By
assumption $T$ is a solution of $L_{\psi(T)}$ and there exists a $y
\in Y$ equivalent to $\psi(T)$, that is $\psi(T)^g=y$ for some $g\in
G$. As $\psi(T)^g=\psi(T^g)$, it follows that $T^g$ is a solution of
$L_y=L_{\psi(T^g)}$.

(ii): Let $T$ be equivalent to $S$. Thus $T^g=S$ for some $g\in
G$. Note that $x=\psi(S)$ and $y=\psi(T)$ as the solution sets of
different CSPs in $\mc{L}$ are disjoint. Hence, $x = \psi(S) = \psi(T^g) =
\psi(T)^g = y^g$, showing that $x$ is equivalent to $y$. 

(iii): Let $g\in G$ be arbitrary. Then $T^g$ is a solution of
$L_{\psi(T^g)}$. Since the CSPs in $\mc{L}$ have disjoint
solution sets, $T^g$ is a solution of $L_x$ if and only if
$\psi(T^g)=x^g=x$. Hence $T^g$ is a solution of $L_x$ if and only if
$g$ lies in the stabiliser of $x$ in $G$.
\qed
\end{proof}

Choosing $\mc{L}=\mc{L}_{\mc{F}_n}$, $\mc{T}$ to be $\Omega_n$, $\phi$
to be the action defined in \eqref{eq_action}, and $\psi$
as the mapping sending multiplication tables to the function
corresponding to their diagonal, satisfies the conditions in
Lemma~\ref{lem_case_split}. To obtain a set of non-equivalent
functions in $\mc{F}_n$ under the induced action $\phi^{\psi}$ is then a
reformulation of a well-known problem: the equivalence classes of
functions are in one-to-one correspondence with unlabelled functional
digraphs, that is directed graphs in which every vertex has outdegree~1. 
The construction of diagonals
and the role they play in the multiplication tables of semigroups is
discussed in detail in \cite[Chapter 3]{Dis10}.  
If $\ol{\mc{F}}_n$ denotes a set of representatives of non-equivalent
functions in $\mc{F}_n$ then each structural type of semigroup appears
as solution of $\mc{L}_{\ol{\mc{F}}_n}$ due to
Lemma~\ref{lem_case_split}(i). Moreover, different CSPs in
$\mc{L}_{\ol{\mc{F}}_n}$ have pairwise not (anti\=/)isomorphic solutions
due to the contraposition of Lemma~\ref{lem_case_split}(ii). This
allows us to search independently in different CSPs for solutions up
to (anti\=/)isomorphism. 

The solutions of $L_f$ form orbits under the
stabiliser of $f$ in $S_n\times C_2$ according to
Lemma~\ref{lem_case_split}(iii). Following the considerations in
Section~\ref{ssec_SB} we add the canonicity constraint
\eqref{const_lex} for every non-identity element in the stabiliser to
$L_f$ to obtain a CSP $\ol{L}_f$ with one solution from every
orbit. Hence the solutions of $\ol{\mc{L}}_{\ol{\mc{F}}_n} =
\{\ol{L}_f \mid f \in \ol{\mc{F}}_n\}$  form a set of semigroups on
$[n]$ up to (anti\=/)isomorphism.
As before we define a CSP $L^{-3}_f$ by adding constraint
\eqref{const_3nil} to $L_f$, ruling out zero semigroups and nilpotent
semigroups of degree $3$. Adding this constraint is not necessary for
all functions $f$, since $L_f$ does not always allow solutions that are
nilpotent of degree $3$. In particular, $f$ must not have more than one
fixed point. The family of CSPs $\{L^{-3}_f\mid f\in
\ol{\mc{F}}_n\}$ is denoted by $\mc{L}^{-3}_{\ol{\mc{F}}_n}$ and
analogously we get 
\begin{equation}
\label{eq_family_sg}
\ol{\mc{L}}^{-3}_{\ol{\mc{F}}_n}=\{\ol{L}^{-3}_f\mid f\in
\ol{\mc{F}}_n\}.
\end{equation}

Calculating the stabiliser in $S_n\times C_2$ of a function
$f$ corresponding to a diagonal directly under the induced action
is not very efficient. This can be avoided by reformulating the
action to a pointwise action on sets.
The reformulation was in principal already introduced in
Section~\ref{ssec_SB}. Every element $g\in S_n\times C_2$ induces a
bijection of the literals of the CSP $L_f$. Take the set of literals
$\chi_f=\{(t_{i,i}=f(i))\mid 1\leq i\leq n\}$ corresponding to the
given diagonal entries. Then $g$ is in the stabiliser of $f$ if
and only if $\chi_f^g=\chi_f$. It is not a coincidence that the
stabiliser of $f$ in $S_n\times C_2$ equals the stabiliser of a set of
literals, as shown by  the following result complementing
Lemma~\ref{lem_case_split}.

\begin{lemma}
\label{lem_SB_lit}
Let $L=(V,D,C)$ be a CSP with non-empty solution set and let $\phi:
(V\times D) \times G \rightarrow V\times D$ be an action on the
literals sending instantiations to instantiations. Denote the setwise
stabiliser of $\chi$ in $G$ by $\mathrm{Stab}_G(\chi)$. 

 If there exists a
subgroup $H \leq G$ such that each set of equivalent solutions of $L$ is
an orbit under $H$, and if there exists a subset of all literals $\chi
\subseteq V\times D$ such that the solutions of $L$ are the subsets
of~$\chi$ that are total instantiations, then each set of equivalent
solutions forms an orbit under $\mathrm{Stab}_G(\chi)$.
\end{lemma}
\begin{proof}
Denote the set of solutions of $L$ by $\mc{T}$. For every $T\in\mc{T}$
and every element $g\in\mathrm{Stab}_G(\chi)$ it follows from
$T^g\subseteq \chi^g = \chi$ that $T^g$ is in $\mc{T}$.

It remains to show that $H \leq \mathrm{Stab}_G(\chi)$. Note that
$\chi$ equals the union of all solutions. Let $h\in H$ then 
\[
\chi^h=\left(\bigcup_{T\in\mc{T}}T\right)^h=\bigcup_{T\in\mc{T}}T^h
=\bigcup_{T\in\mc{T}}T=\chi
\]
and hence $h\in\mathrm{Stab}_G(\chi)$.
\qed
\end{proof}

Lemma \ref{lem_SB_lit} does not directly apply to the CSPs in
$\mc{L}_{\mc{F}_n}$, because of the associativity constraint. If one
neglects this constraint, such that the solutions are all tables
fulfilling the remaining constraints, then the assumptions of
Lemma~\ref{lem_SB_lit} are satisfied. Any total instantiation for
which the values on the diagonal are in $\chi_f$ is a solution for
$f\in \mc{F}_n$, and equivalent solutions form orbits under the
stabiliser of the literals in $S_n\times C_2$. Adding the
associativity constraint back in does not change this fact, since
associativity is invariant under isomorphism and anti-isomorphism.
 
Having a family of CSPs depending on the diagonal is not enough
to resolve the computational bottleneck mentioned at the end of
Section~\ref{ssec_SB}.
The number of constraints is
still  $2n!-1$ if $f$ equals the identity function on $[n]$.
In the next section we show how to avoid this problem by
applying the technique from Lemma~\ref{lem_case_split} again.

\section{Enumeration of bands}
\label{sec_bands}

If $f$ is the identity function $\id_n$ on $[n]$ then the solutions of
$L_f$ as defined in CSP \ref{csp_diag} are the bands on $[n]$. The
structure of bands is well understood, knowledge that we shall use
together with Lemma~\ref{lem_case_split} to substitute $L_{\id_n}$
with a family of CSPs.

For the search we rely on a classification of rectangular
bands. Every rectangular band is isomorphic to 
a semigroup on a Cartesian product $I \times \Lambda$ with
multiplication defined by $(i,\lambda)(j,\mu)=(i,\mu)$, and each such
multiplication defines a rectangular band. Two rectangular bands
$I_1\times \Lambda_1$ and $I_2\times \Lambda_2$ are isomorphic if and
only if $|I_1|=|I_2|$ and $|\Lambda_1|=|\Lambda_2|$, and they are
anti-isomorphic if and only if $|I_1|=|\Lambda_2|$ and
$|\Lambda_1|=|I_2|$. Hence, the number of rectangular bands on $[n]$
up to (anti\=/)isomorphism equals the
number of divisors of $n$ that are less than or equal to $\sqrt{n}$. 

Every band is a semilattice of rectangular bands~\cite{Cli41}. To
define a family of CSPs we use the trivial consequence that the minimal
$\mc{D}$-class of a band is a rectangular band.

\begin{csp}
Given a rectangular band $R\subseteq [n]$ define a CSP
$B_R=(V_n,D_n,C_R)$ based on $L_{\mathrm{id}_n}=(V_n,D_n,C_{\id_n})$
by adding the constraints
\begin{eqnarray}
\label{eq_const_bandsI}t_{i,j}=ij & \mathrm{if} & i,j\in R\\
\label{eq_const_bandsII}t_{i,j},t_{j,i}\in R & \mathrm{if} &i\in R, j\in [n]
\end{eqnarray}
to $C_{\id_n}$ to obtain $C_R$.
\end{csp}

It is obvious that the multiplication table of every band with $R$ as
minimal $\mc{D}$-class is a solution of $B_R$. Given a solution of
$B_R$ all elements of $R$ in the corresponding band are $\mc{D}$-related
due to constraint \eqref{eq_const_bandsI} and are in the minimal
$\mc{D}$-class due to constraint \eqref{eq_const_bandsII}. Elements in
the complement of $R$ are not $\mc{D}$-related to elements in $R$, as
constraint \eqref{eq_const_bandsII} implies that their two-sided
ideals differ. Consequently, the solutions of $B_R$
are exactly the bands on $[n]$ having $R$ as their minimal
$\mc{D}$-class.
 
Let $\mc{R}^k_n$ denote the rectangular
bands on all subsets of $[n]$ of size $k$ and let
$\mc{R}_n=\cup_{k=1}^n\mc{R}^k_n$. We then define the family of CSPs
$\mc{L}_{\mc{R}_n} = \{B_R \mid R \in \mc{R}_n\}$ which fulfils the
conditions of Lemma~\ref{lem_case_split}. It follows that each
(anti\=/)isomorphism type of band will appear as a solution of exactly
one of the CSPs in $\mc{L}_{\ol{\mc{R}}_n} = \{B_R \mid R \in \ol{\mc{R}}_n\}$
where $\ol{\mc{R}}_n$ denotes a set of representatives of rectangular
bands of order at most $n$ up to (anti\=/)isomorphism.

Lemma~\ref{lem_SB_lit} allows us to compute the symmetries of a CSP
$B_R$  as a stabiliser of literals. We see that an element in
$S_n\times C_2$ is a symmetry if and only if its restriction to $R$
is an automorphism or anti-automorphism.  Adding the canonicity
constraint~\eqref{const_lex} for every non-identity element in the
symmetry group then yields $\ol{B}_R$. The number of constraints
added is maximal when $R$ is the left (or right) zero semigroup on
$[n]$, but then $R$ is the unique solution of $B_R$ because constraint
\eqref{eq_const_bandsI} covers the whole multiplication table. As no
actual search is needed in this case, replacing $\ol{L}_{\id_n}$ with
the family of CSPs
\begin{equation}
\label{eq_family_bands}
\ol{\mc{L}}_{\mc{\ol{R}}_n} = \{\ol{B}_R \mid R \in \mc{\ol{R}}_n\}
\end{equation}
strictly reduces the number of symmetries involved, thereby reducing
the effect of the computational bottleneck discussed at the ends of
Sections~\ref{ssec_SB} and~\ref{sec_cases}.

\section{The semigroups of order 9}
\label{sec_sg9}

We have used the families of CSPs introduced in the previous sections
to obtain canonical representatives for semigroups of order $9$ which
are not nilpotent of degree $3$ up to (anti\=/)isomorphism. More
precisely, we solved the CSPs in
\begin{equation}
\label{eq_families}
\ol{\mc{L}}^{-3}_{\ol{\mc{F}}_9}\setminus \left\{\ol{L}_{\id_9}\right\} 
 \mathrm{~and~} \ol{\mc{L}}_{\mc{\ol{R}}_9}, 
\end{equation}
as defined in \eqref{eq_family_sg} and \eqref{eq_family_bands}
obtaining a total of $23\,161\,651\,504$ solutions. The
semigroups not searched for were the zero semigroup and the
nilpotent semigroups of degree $3$ of order $9$. The number of the
latter is $52\,966\,239\,062\,973$~\cite[Table
  4]{DM12}. All together there are $52\,989\,400\,714\,478$ semigroups
of order $9$ up to (anti\=/)isomorphism of which almost $99.96 \%$ are
nilpotent of degree 3.

To perform the computations we used \textsf{GAP}~\cite{GAP4} and
\textsf{Minion}~\cite{GJM06}; the former to calculate
stabilisers and also for the automated creation of the input files;
the latter to solve the CSPs. The computations took
around 87 hours on a machine with 2.66\,GHz Intel X-5430 processor and
8\,GB RAM. The code can be found in~\cite[Appendix C]{Dis10}.

\subsection{Classification}
\label{ssec:class}
\begin{table}
\caption{\label{tab_by_idem}Numbers of semigroups on
  $[n]$ up to (anti\=/)isomorphism}
{\small
\begin{tabular}{crrrrrrrr}
\toprule
{\bf n}
&\multicolumn{1}{c}{\bf 2}
&\multicolumn{1}{c}{\bf 3}
&\multicolumn{1}{c}{\bf 4}
&\multicolumn{1}{c}{\bf 5}
&\multicolumn{1}{c}{\bf 6}
&\multicolumn{1}{c}{\bf 7}
&\multicolumn{1}{c}{\bf 8}
&\multicolumn{1}{c}{\bf 9}\\
\midrule
\#
& 4 & 18 & 126 & 1\,160 & 15\,973 & 836\,021 & 1\,843\,120\,128 
& 52\,989\,400\,714\,478\\
\midrule
$e$ & \multicolumn{8}{c}{{\it by number $e$ of idempotents}}\\
\midrule
1&
2 & 5 & 19 & 132 & 3\,107 & 623\,615 & 1\,834\,861\,133 &
52\,976\,551\,026\,562\\ 
2&
2 & 7 & 37 & 216 & 1\,780 & 32\,652 & 4\,665\,709 & 12\,710\,266\,442\\
3&
& 6 & 44 & 351 & 3\,093 & 33\,445 & 600\,027 & 68\,769\,167\\ %
4&
&& 26 & 326 & 4\,157 & 53\,145 & 754\,315 & 14\,050\,493\\ %
5&
&&& 135 & 2\,961 & 56\,020 & 1\,007\,475 & 18\,660\,074\\ %
6&
&&&& 875 & 30\,395 & 822\,176 & 20\,044\,250\\ %
7&
&&&&& 6\,749 & 348\,692 & 12\,889\,961\\ %
8&
&&&&&& 60\,601 & 4\,389\,418\\ %
9&
&&&&&&& 618\,111\\ %
\midrule
$d$ & \multicolumn{8}{c}{{\it by minimal generator number $d$}}\\
\midrule
1&
2 & 3 & 4 & 5 & 6 & 7 & 8 & 9\\
2&
2 & 11 & 48 & 149 & 441 & 1\,230 & 3\,464 & 9\,945\\
3&
& 4 & 65 & 588 & 4\,506 & 27\,743 & 156\,898 & 911\,672\\
4&
&& 9 & 397 & 8\,370 & 549\,037 & 18\,014\,631 & 240\,061\,550\\
5&
&&& 21 & 2\,600 & 239\,410 & 1\,774\,277\,445 & 791\,830\,876\,983\\
6&
&&&& 50 & 18\,474 & 50\,525\,311 & 52\,140\,869\,887\,616\\
7&
&&&&& 120 & 142\,082 & 56\,457\,790\,001\\
8&
&&&&&& 289 & 1\,176\,005\\
9&
&&&&&&& 697\\
\bottomrule
\end{tabular}
}
\end{table}

We analysed the semigroups obtained by search to extract various
classification results.  
The numbers of semigroups with $9$ elements sorted by their number of
idempotents are listed in Table~\ref{tab_by_idem} together with
numbers for lower orders from~\cite[Table 4.1]{SYT94},
which we also verified using our search method. Also listed in the
table are numbers of semigroups by their minimal generator number. For
nilpotent semigroups of degree 3 these numbers are easily calculated
from the summands of the formula given in \cite[Theorem 2.3]{DM12}
using the fact that every nilpotent semigroup is generated by its
indecomposable elements.

Information on the classification of semigroups of order 9 by certain
properties is summarised in Table~\ref{tab_sg_9}. The
total number of commutative semigroups agrees with the result
from~\cite{Gri96}. The selection of properties was largely inspired
by~\cite[Table 4.2]{SYT94} except that we also report
numbers of self-dual semigroups. We determined the latter using the
method described in the second paragraph of Section~\ref{ssec:autos}
to the CSPs from \eqref{eq_families}. In addition we needed the
number of self-dual semigroups of order 9 that are nilpotent of degree
$3$ which is $606\,097\,491$~\cite[Table 5]{DM12}. Note that except
for the regular semigroups all classes listed in Table~\ref{tab_sg_9}
consist entirely of self-dual semigroups.

\begin{table}
\caption{\label{tab_sg_9}Numbers of semigroups of order 9 with various
properties}
\begin{tabular}{crrrrr}
\toprule
{\bf Idpts} & self-dual\hfill \ & commutative\hfill \ & regular\hfill \ &
inverse\hfill \ & comm.-inv.\hfill \ \\
\midrule
{\bf 1} & 613\,365\,656 & 9\,940\,825 & 2 & 2 & 2 \\
{\bf 2} & 8\,265\,721 & 664\,080 & 23 & 23 & 16 \\
{\bf 3} & 739\,317 & 249\,330 & 148 & 129 & 111 \\
{\bf 4} & 410\,158 & 222\,637 & 830 & 567 & 504 \\
{\bf 5} & 328\,937 & 201\,060 & 4\,136 & 1\,750 & 1\,555 \\
{\bf 6} & 223\,226 & 148\,647 & 17\,535 & 3\,870 & 3\,460 \\
{\bf 7} & 113\,160 & 82\,481 & 66\,822 & 6\,582 & 6\,137 \\
{\bf 8} & 38\,979 & 30\,789 & 217\,437 & 7\,505 & 7\,505 \\
{\bf 9} & 7\,510 & 5\,994 & 618\,111 & 5\,994 & 5\,994\\
\midrule
$\mathbf{\sum}$ & 623\,492\,664 & 11\,545\,843 & 925\,044 & 26\,422 &
25\,284 \\
\bottomrule
\end{tabular}
\end{table} 

\subsection{Up to isomorphism}
As mentioned in the introduction we also obtained results for the
classification of semigroups up to isomorphism. In general this is
achieved by replacing the group $S_n\times C_2$ wherever it appears in
the considerations regarding symmetries of the CSPs with the group
$S_n$. In many situations we can alternatively take advantage of the
fact that we determined self-dual semigroups: twice the number of
semigroups up to (anti\=/)isomorphism minus the number of
self-dual semigroups yields the number of semigroups up to
isomorphism. Hence there are $105\,978\,177\,936\,292$
semigroups of order 9 up to isomorphism. These semigroups together
with those of lower orders are classified by their number of
idempotents and by their minimal generator number in
Table~\ref{tab_by_idem_iso}.

\begin{table}
\caption{\label{tab_by_idem_iso}Numbers of semigroups on $[n]$ up to
  isomorphism}
{\small
\begin{tabular}{crrrrrrrr}
\toprule
{\bf n}
&\multicolumn{1}{c}{\bf 2}
&\multicolumn{1}{c}{\bf 3}
&\multicolumn{1}{c}{\bf 4}
&\multicolumn{1}{c}{\bf 5}
&\multicolumn{1}{c}{\bf 6}
&\multicolumn{1}{c}{\bf 7}
&\multicolumn{1}{c}{\bf 8}
&\multicolumn{1}{c}{\bf 9}\\
\midrule
\#
& 5 & 24 & 188 & 1\,915 & 28\,634 & 1\,627\,672 & 3\,684\,030\,417
& 105\,978\,177\,936\,292\\
\midrule
$e$ & \multicolumn{8}{c}{{\it by number $e$ of idempotents}}\\
\midrule
1&
2 & 5 & 20 & 171 & 5\,284 & 1\,224\,331 & 3\,667\,785\,000
& 105\,952\,488\,687\,468\\
2&
3 & 9 & 50 & 309 & 2\,806 & 58\,583 & 9\,207\,430 &  25\,412\,267\,163\\
3&
& 10 & 72 & 590 & 5\,422 & 61\,323 & 1\,150\,085 & 136\,799\,017\\
4&
&& 46 & 594 & 7\,772 & 101\,539 & 1\,466\,691 & 27\,690\,828 \\
5&
&&& 251 & 5\,668 & 109\,107 & 1\,983\,558 & 36\,991\,211\\
6&
&&&& 1\,682 & 59\,576 & 1\,626\,956 & 39\,865\,274\\
7&
&&&&& 13\,213 & 690\,871 & 25\,666\,762\\
8&
&&&&&& 119\,826 & 8\,739\,857\\
9&
&&&&&&& 1\,228\,712\\
\midrule
$d$ & \multicolumn{8}{c}{{\it by minimal generator number $d$}}\\
\midrule
1&
2 & 3 & 4 & 5 & 6 & 7 & 8 & 9\\
2&
3 & 14 & 64 & 212 & 664 & 1\,930 & 5\,678 & 17\,010\\
3&
& 7 & 103 & 954 & 7\,835 & 50\,541 & 294\,622 & 1\,751\,293\\
4&
&& 17 & 703 & 15\,144 & 1\,075\,353 & 35\,850\,090 & 479\,050\,352\\
5&
&&& 41 & 4\,886 & 463\,784 & 3\,546\,839\,307 & 1\,583\,613\,947\,364\\
6&
&&&& 99 & 35\,818 & 100\,760\,203 & 104\,281\,178\,828\,643\\
7&
&&&&& 239 & 279\,932 & 112\,902\,004\,698\\
8&
&&&&&& 577 & 2\,335\,530\\
9&
&&&&&&& 1\,393\\
\bottomrule
\end{tabular}
}
\end{table}

\section{Automorphism groups and distinct semigroups on a set}
\label{sec_autos}

To determine the automorphism groups of semigroups with at most $9$
elements, we use the idea described in the first paragraph of
Section~\ref{ssec:autos}.
Technically there are two different methods: we can record for each
semigroup the isomorphisms that are automorphisms, or we can perform
one search for every possible automorphism group. Neither approach is
by itself feasible for $n=9$, because of the large numbers of
semigroups with $9$ elements and of subgroups of $S_9$. We therefore
took a mixed approach distinguishing the
following mutually exclusive cases depending on the orders of
automorphisms a semigroup $S$ allows.

\begin{enumerate}[(i)]
\item $\mbox{Aut}(S)\cong C^k_2$, $k\in\N$:
  Require strict inequality in all constraints that do not correspond
  to a permutation of order 2. Require also that equality holds for exactly
  $2^k-1$ of the remaining  constraints.

\item $|\mbox{Aut}(S)|=2^k$, $k\in\N$,
      but $\mbox{Aut}(S)\not\cong C_2^k$: 
  Require strict inequality in all constraints that do not
  correspond to a permutation of order $2^m, m\in \N$. Require also that
  equality holds for at least one constraint corresponding
  to a permutation of order 4.

\item $|\mbox{Aut}(S)|$ contains an odd prime factor:
  Require that equality
  holds for at least one constraint corresponding to a permutation of
  odd prime order.
\end{enumerate}
Semigroups covered by none of the three cases must have the
trivial group as automorphism group. In the first case $\lfloor
n/2\rfloor$ is an upper bound for $k$ because $C^k_2$ is a subgroup of
$S_n$ if and only if $2k\leq n$ 
(see \cite[Theorem 2]{Joh71}). We run a separate computation for each
admissible value of $k$, specifying a unique isomorphism type of
automorphism group, and record the number of solutions. In the other
two cases we let \textsf{Minion} output a list of automorphisms for
each solution and read it into {\sf GAP}. We then use the
identification function in the {\sf SmallGroups}
library~\cite{SmallGroups} to find the isomorphism types of the
groups. Note that it was not possible to exclude nilpotent semigroups
of degree 3 from the various searches as their
numbers with prescribed automorphism groups are unknown.

\begin{table}
\caption{\label{aut_2}Automorphism groups of semigroups of order 2}
\begin{tabular}{ccrr}
\toprule
automorphism & ID & \multicolumn{1}{c}{number up to} &\multicolumn{1}{c}{number up to}\\ 
group & & \multicolumn{1}{c}{(anti\=/)isomorphism} & \multicolumn{1}{c}{isomorphism}\\
\midrule
trivial & $(1,1)$ & 3 & 3\\
$C_2$ & $(2,1)$ & 1 & 2\\
\bottomrule
\end{tabular}
\end{table}

\begin{table}
\caption{\label{aut_3}Automorphism groups of semigroups of order 3}
\begin{tabular}{ccrr}
\toprule
automorphism & ID & \multicolumn{1}{c}{number up to} &\multicolumn{1}{c}{number up to}\\ 
group & & \multicolumn{1}{c}{(anti\=/)isomorphism} & \multicolumn{1}{c}{isomorphism}\\
\midrule
trivial & $(1,1)$ & 12 & 15\\
$C_2$ & $(2,1)$ & 5 & 7\\
$S_3$ & $(6,1)$ & 1 & 2\\
\bottomrule
\end{tabular}
\end{table}

\begin{table}
\caption{\label{aut_4}Automorphism groups of semigroups of order 4}
\begin{tabular}{ccrr}
\toprule
automorphism & ID & \multicolumn{1}{c}{number up to} &\multicolumn{1}{c}{number up to}\\ 
group & & \multicolumn{1}{c}{(anti\=/)isomorphism} & \multicolumn{1}{c}{isomorphism}\\
\midrule
trivial & $(1,1)$ & 78 & 112\\
$C_2$ & $(2,1)$ & 39 & 62\\
$C_2 \times C_2$ & $(4,2)$ & 3 & 5\\
$S_3$ & $(6,1)$ & 5 & 7\\
$S_4$ & $(24,12)$ & 1 & 2\\
\bottomrule
\end{tabular}
\end{table}

\begin{table}
\caption{\label{aut_5}Automorphism groups of semigroups of order 5}
\begin{tabular}{ccrr}
\toprule
automorphism & ID & \multicolumn{1}{c}{number up to} &\multicolumn{1}{c}{number up to}\\ 
group & & \multicolumn{1}{c}{(anti\=/)isomorphism} & \multicolumn{1}{c}{isomorphism}\\
\midrule
trivial & $(1,1)$ & 746 & 1221\\
$C_2$ & $(2,1)$ & 342 & 576\\
$C_3$ & $(3,1)$ & 2 & 2\\
$C_4$ & $(4,1)$ & 1 & 1\\
$C_2 \times C_2$ & $(4,2)$ & 26 & 46\\
$S_3$ & $(6,1)$ & 33 & 51\\
$D_8$ & $(8,3)$ & 1 & 2\\
$D_{12}$ & $(12,4)$ & 4 & 8\\
$S_4$ & $(24,12)$ & 4 & 6\\
$S_5$ & $(120,34)$ & 1 & 2\\
\bottomrule
\end{tabular}
\end{table}

\begin{table}
\caption{\label{aut_6}Automorphism groups of semigroups of order 6}
\begin{tabular}{ccrr}
\toprule
automorphism & ID & \multicolumn{1}{c}{number up to} &\multicolumn{1}{c}{number up to}\\ 
group & & \multicolumn{1}{c}{(anti\=/)isomorphism} & \multicolumn{1}{c}{isomorphism}\\
\midrule
trivial & $(1,1)$ & 10\,965 & 19\,684\\
$C_2$ & $(2,1)$ & 4\,121 & 7\,397\\
$C_3$ & $(3,1)$ & 26 & 32\\
$C_4$ & $(4,1)$ & 7 & 7\\
$C_2 \times C_2$ & $(4,2)$ & 441 & 806\\
$S_3$ & $(6,1)$ & 300 & 506\\
$D_8$ & $(8,3)$ & 17 & 30\\
$C_2 \times C_2 \times C_2$ & $(8,5)$ & 6 & 12\\
$D_{12}$ & $(12,4)$ & 49 & 92\\
$S_4$ & $(24,12)$ & 30 & 48\\
$S_3 \times S_3$ & $(36,10)$ & 2 & 4\\
$C_2 \times S_4$ & $(48,48)$ & 4 & 8\\
$S_5$ & $(120,34)$ & 4 & 6\\
$S_6$ & $(720,763)$ & 1 & 2\\
\bottomrule
\end{tabular}
\end{table}

\begin{table}
\caption{\label{aut_7}Automorphism groups of semigroups of order 7}
\begin{tabular}{ccrr}
\toprule
automorphism & ID or & \multicolumn{1}{c}{number up to} &\multicolumn{1}{c}{number up to}\\ 
group & order & \multicolumn{1}{c}{(anti\=/)isomorphism} & \multicolumn{1}{c}{isomorphism}\\
\midrule
trivial & $(1,1)$ & 746\,277 & 1\,458\,882\\
$C_2$ & $(2,1)$ & 76\,704 & 144\,879\\
$C_3$ & $(3,1)$ & 412 & 620\\
$C_4$ & $(4,1)$ & 82 & 101\\
$C_2 \times C_2$ & $(4,2)$ & 7\,314 & 13\,756\\
$C_5$ & $(5,1)$ & 6 & 6\\
$S_3$ & $(6,1)$ & 3\,638 & 6\,552\\
$C_6$ & $(6,2)$ & 37 & 53\\
$C_4 \times C_2$ & $(8,2)$ & 4 & 6\\
$D_8$ & $(8,3)$ & 169 & 282\\
$C_2 \times C_2 \times C_2$ & $(8,5)$ & 172 & 330\\
$D_{10}$ & $(10,1)$ & 2 & 2\\
$D_{12}$ & $(12,4)$ & 790 & 1\,476\\
$C_2 \times D_8$ & $(16,11)$ & 10 & 20\\
$S_4$ & $(24,12)$ & 277 & 475\\
$C_2 \times C_2 \times S_3$ & $(24,14)$ & 14 & 28\\
$S_3 \times S_3$ & $(36,10)$ & 24 & 44\\
$C_2 \times S_4$ & $(48,48)$ & 45 & 86\\
$(S_3 \times S_3) : C_2$ & $(72,40)$ & 1 & 2\\
$S_5$ & $(120,34)$ & 30 & 48\\
$S_3 \times S_4$ & $(144,183)$ & 4 & 8\\
$C_2 \times S_5$ & $(240,189)$ & 4 & 8\\
$S_6$ & $(720,763)$ & 4 & 6\\
$S_7$ & $5040$ & 1 & 2\\
\bottomrule
\end{tabular}
\end{table}

\begin{table}
\caption{\label{aut_8}Automorphism groups of semigroups of order 8}
\begin{tabular}{ccrr}
\toprule
automorphism & ID or & \multicolumn{1}{c}{number up to} &\multicolumn{1}{c}{number up to}\\ 
group & order & \multicolumn{1}{c}{(anti\=/)isomorphism} & \multicolumn{1}{c}{isomorphism}\\
\midrule
trivial & $(1,1)$ & 1\,834\,638\,770 & 3\,667\,253\,972\\
$C_2$ & $(2,1)$ & 8\,176\,697 & 16\,194\,638\\
$C_3$ & $(3,1)$ & 17\,297 & 31\,567\\
$C_4$ & $(4,1)$ & 1\,270 & 1\,907\\
$C_2 \times C_2$ & $(4,2)$ & 188\,316 & 363\,902\\
$C_5$ & $(5,1)$ & 92 & 110\\
$S_3$ & $(6,1)$ & 69\,275 & 131\,242\\
$C_6$ & $(6,2)$ & 1\,249 & 2\,086\\
$C_4 \times C_2$ & $(8,2)$ & 105 & 153\\
$D_8$ & $(8,3)$ & 2\,238 & 3\,876\\
$C_2 \times C_2 \times C_2$ & $(8,5)$ & 5\,324 & 10\,255\\
$C_3 \times C_3$ & $(9,2)$ & 5 & 6\\
$D_{10}$ & $(10,1)$ & 28 & 34\\
$D_{12}$ & $(12,4)$ & 13\,583 & 25\,883\\
$C_2 \times D_8$ & $(16,11)$ & 263 & 490\\
$C_2 \times C_2 \times C_2 \times C_2$ & $(16,14)$ & 15 & 29\\
$C_3 \times S_3$ & $(18,3)$ & 40 & 56\\
$C_5 : C_4$ & $(20,3)$ & 1 & 1\\
$C_4 \times S_3$ & $(24,5)$ & 4 & 6\\
$S_4$ & $(24,12)$ & 3\,461 & 6\,293\\
$C_2 \times A_4$ & $(24,13)$ & 4 & 4\\
$C_2 \times C_2 \times S_3$ & $(24,14)$ & 491 & 966\\
$S_3 \times S_3$ & $(36,10)$ & 368 & 674\\
$D_8 \times S_3$ & $(48,38)$ & 11 & 22\\
$C_2 \times S_4$ & $(48,48)$ & 768 & 1\,445\\
$(S_3 \times S_3) : C_2$ & $(72,40)$ & 16 & 28\\
$C_2 \times S_3 \times S_3$ & $(72,46)$ & 12 & 24\\
$C_2 \times C_2 \times S_4$ & $(96,226)$ & 14 & 28\\
$S_5$ & $(120,34)$ & 277 & 475\\
$S_3 \times S_4$ & $(144,183)$ & 44 & 84\\
PSL$(3,2)$ & $(168,42)$ & 1 & 1\\
$C_2 \times S_5$ & $(240,189)$ & 44 & 84\\
$S_4 \times S_4$ & $(576,8653)$ & 2 & 4\\
$S_6$ & $(720,763)$ & 30 & 48\\
$S_5 \times S_3$ & $(720,767)$ & 4 & 8\\
$C_2 \times S_6$ & $(1440,5842)$ & 4 & 8\\
$S_7$ & $5040$ & 4 & 6\\
$S_8$ & $40320$ & 1 & 2\\
\bottomrule
\end{tabular}
\end{table}

\begin{table}
\caption{\label{aut_9}Automorphism groups of semigroups of order 9}
\begin{tabular}{ccrr}
\toprule
automorphism & ID or & \multicolumn{1}{c}{number up to} &\multicolumn{1}{c}{number up to}\\ 
group & order & \multicolumn{1}{c}{(anti\=/)isomorphism} & \multicolumn{1}{c}{isomorphism}\\
\midrule
trivial & $(1,1)$ & 52\,961\,873\,362\,324 & 105\,923\,135\,799\,007\\
$C_2$ & $(2,1)$ & 27\,478\,363\,462 & 54\,944\,831\,554\\
$C_3$ & $(3,1)$ & 6\,329\,218 & 12\,562\,447\\
$C_4$ & $(4,1)$ & 53\,591 & 97\,613\\
$C_2 \times C_2$ & $(4,2)$ & 33\,882\,706 & 67\,399\,096\\
$C_5$ & $(5,1)$ & 1\,547 & 2\,295\\
$S_3$ & $(6,1)$ & 7\,886\,998 & 15\,634\,673\\
$C_6$ & $(6,2)$ & 94\,521 & 180\,353\\
$C_7$ & $(7,1)$ & 18 & 18\\
$C_4 \times C_2$ & $(8,2)$ & 3\,286 & 5\,478\\
$D_8$ & $(8,3)$ & 59\,672 & 110\,744\\
$C_2 \times C_2 \times C_2$ & $(8,5)$ & 203\,597 & 396\,962\\
$C_3 \times C_3$ & $(9,2)$ & 291 & 449\\
$D_{10}$ & $(10,1)$ & 420 & 626\\
$C_{10}$ & $(10,2)$ & 108 & 156\\
$C_{12}$ & $(12,2)$ & 26 & 34\\
$A_4$ & $(12,3)$ & 3 & 3\\
$D_{12}$ & $(12,4)$ & 354\,352 & 689\,994\\
$C_6 \times C_2$ & $(12,5)$ & 850 & 1\,496\\
$D_{14}$ & $(14,1)$ & 4 & 4\\
$C_4 \times C_2 \times C_2$ & $(16,10)$ & 18 & 32\\
$C_2 \times D_8$ & $(16,11)$ & 5\,530 & 10\,252\\
$C_2 \times C_2 \times C_2 \times C_2$ & $(16,14)$ & 1\,345 & 2\,654\\
$C_3 \times S_3$ & $(18,3)$ & 1\,286 & 2\,135\\
$(C_3 \times C_3) : C_2$ & $(18,4)$ & 1 & 2\\
$C_5 : C_4$ & $(20,3)$ & 8 & 9\\
$D_{20}$ & $(20,4)$ & 36 & 52\\
$C_7 : C_3$ & $(21,1)$ & 2 & 2\\
$C_4 \times S_3$ & $(24,5)$ & 105 & 153\\
$C_3 \times D_8$ & $(24,10)$ & 26 & 36\\
$S_4$ & $(24,12)$ & 67\,321 & 128\,046\\
$C_2 \times A_4$ & $(24,13)$ & 57 & 69\\
$C_2 \times C_2 \times S_3$ & $(24,14)$ & 15\,150 & 29\,589\\
$C_4 \times D_8$ & $(32,25)$ & 1 & 2\\
$(C_2 \times C_2 \times C_2 \times C_2) : C_2$ & $(32,27)$ & 10 & 19\\
$C_2 \times C_2 \times D_8$ & $(32,46)$ & 83 & 166\\
$S_3 \times S_3$ & $(36,10)$ & 6\,429 & 12\,123\\
GL$(2,3)$ & $(48,29)$ & 1 & 1\\
$D_8 \times S_3$ & $(48,38)$ & 263 & 486\\
$C_2 \times S_4$ & $(48,48)$ & 13\,204 & 25\,243\\
$C_2 \times C_2 \times C_2 \times S_3$ & $(48,51)$ & 44 & 88\\
$D_8 \times D_8$ & $(64,226)$ & 1 & 1\\
$(S_3 \times S_3) : C_2$ & $(72,40)$ & 158 & 263\\
$C_3 \times S_4$ & $(72,42)$ & 34 & 50\\
$C_2 \times S_3 \times S_3$ & $(72,46)$ & 474 & 940\\
$C_4 \times S_4$ & $(96,186)$ & 4 & 6\\
$C_2 \times C_2 \times S_4$ & $(96,226)$ & 479 & 946\\
$S_5$ & $(120,34)$ & 3\,454 & 6\,281\\
$S_3 \times S_4$ & $(144,183)$ & 705 & 1\,327\\
$C_2 \times ((S_3 \times S_3) : C_2)$ & $(144,186)$ & 11 & 22\\
PSL$(3,2)$ & $(168,42)$ & 3 & 3\\
$D_8 \times S_4$ & $(192,1472)$ & 10 & 20\\
$S_3 \times S_3 \times S_3$ & $(216,162)$ & 4 & 8\\
$C_2 \times S_5$ & $(240,189)$ & 755 & 1\,423\\
$C_2 \times S_3 \times S_4$ & $(288,1028)$ & 24 & 48\\
$((((C_2 \times D_8) : C_2) : C_3) : C_2) : C_2$ & $(384,5602)$ & 1 & 2\\
$C_2 \times C_2 \times S_5$ & $(480,1186)$ & 14 & 28\\
$S_4 \times S_4$ & $(576,8653)$ & 20 & 38\\
$S_6$ & $(720,763)$ & 277 & 475\\
$S_5 \times S_3$ & $(720,767)$ & 44 & 84\\
$(S_4 \times S_4) : C_2$ & $(1152,157849)$ & 1 & 2\\
$C_2 \times S_6$ & $(1440,5842)$ & 44 & 84\\
$S_5 \times S_4$ & 2880 & 4 & 8\\
$S_6 \times S_3$ & 4320 & 4 & 8\\
$S_7$ & 5040 & 30 & 48\\
$C_2 \times S_7$ & 10080 & 4 & 8\\
$S_8$ & 40320 & 4 & 6\\
$S_9$ & 362880 & 1 & 2\\
\bottomrule
\end{tabular}
\end{table}

Tables~\ref{aut_2}, \ref{aut_3}, \ref{aut_4}, \ref{aut_5},
\ref{aut_6}, \ref{aut_7}, \ref{aut_8},
and~\ref{aut_9} list the automorphism groups of
semigroups up to (anti\=/)isomorphism and up to isomorphism with 2 to 9
elements. There is one table
for each order, containing one line for each isomorphism type of
automorphism group. The groups are identified by their ID in the {\sf
  SmallGroups} library~\cite{SmallGroups}, if their order is less than
2000. In all cases a structural description, computed using the {\sf
  GAP} command {\tt StructureDescription}, is also given.  Finally, the
numbers of semigroups up to (anti\=/)isomorphism and up to isomorphism 
with the given group as automorphism group are provided. The numbers
for semigroups up to (anti\=/)isomorphism of order at most 7 agree with
those from~\cite{ABMN10}, except for an obviously typographic omission
of $C_2\times C_2$ as automorphism group for semigroups of order
5. The numbers for semigroups up to (anti\=/)isomorphism of order 9
partially differ from those in \cite[Table A.15]{Dis10}, where some
semigroups belonging to Case (iii) above were incorrectly counted as
having trivial automorphism group.

The {\sf Minion} computations to obtain the results took nearly
two months on our machine with 2.66\,GHz Intel X-5430
processor. To reduce the possibility of an error 
 we confirmed the numbers in a
second run using a different setup. Details about the code used to
compute the automorphism groups can be found in~\cite[Appendix
  C.2.3]{Dis10}.

An immediate observation is that most of the semigroups have trivial
automorphism group and their ratio to all semigroups seems
to converge to 1 with increasing order, thus supporting a conjecture
from~\cite{DM12}.

We further use our results to deduce the numbers of distinct semigroups
on sets with 2 to 9 elements. For a semigroup $S$ the number of
isomorphic semigroups on the same underlying set equals
$|S|!/|\operatorname{Aut}(S)|$. Hence the number of distinct
semigroups on a set with $n$ elements equals
\begin{equation*}
n!\sum_{S} \frac{1}{|\operatorname{Aut}(S)|}
\end{equation*}
where the summation runs over a set of representatives of semigroups
of order $n$ up to isomorphism. New results in
Table~\ref{tab_all} are the numbers for orders 8 and 9, for lower
orders we confirm the numbers available from \cite[Sequence
  A023814]{OEIS}.

\begin{table}
\caption{\label{tab_all}Numbers of distinct semigroups on $[n]$}
\label{all_semi}
\begin{tabular}{lr}
\toprule
$n$ & \multicolumn{1}{c}{semigroups on $[n]$}\\
\midrule
2 & 8\\
3 & 113\\
4 & 3\,492\\
5 & 183\,732\\
6 & 17\,061\,118\\
7 & 7\,743\,056\,064\\
8 & 148\,195\,347\,518\,186\\
9 & 38\,447\,365\,355\,811\,944\,462\\
\bottomrule
\end{tabular}
\end{table}

While it is known that every group appears as the automorphism group
of some semigroup, the number of isomorphism types of automorphism
groups is small in comparison with the number of all isomorphism
types of subgroups of the symmetric group  
(Table~\ref{tab_autgrp_isos}). Information about which types of
automorphism groups appear could be useful in the development of
algorithms to compute the automorphism group of a given semigroup. 
We observe in particular that for $2\leq n\leq 9$ only
rectangular bands have a transitive subgroup of $S_n$ as automorphism
group. We complete this section by showing that this statement holds
for every order.

\begin{table}
\caption{\label{tab_autgrp_isos}Comparison of the numbers of
  isomorphism types of (a) subgroups of the symmetric group of degree
  $n$ with (b) automorphism groups of semigroups of order $n$}
{\small
\begin{tabular}{lrrrrrrrrr}
\toprule
$n$ 
& 1 & 2 & 3 & 4 & 5 & 6 & 7 & 8 & 9 \\
\midrule
(a) subgroups of $S_n$
& 1 & 2 & 4 & 9 & 16 & 29 & 55 & 137 & 241\\
(b) $\mbox{Aut}(S)$ for $|S|=n$
& 1 & 2 & 3 & 5 & 10 & 14 & 24 & 38 & 65\\
\bottomrule
\end{tabular}
}
\end{table} 

\begin{proposition}
\label{prop_autgrp}
Let $R$ be a finite semigroup. Then the following are equivalent:
\begin{enumerate}[(i)]
\item $R$ is a rectangular band.
\item The automorphism group of $R$ acts transitively on $R$.
\end{enumerate}
\end{proposition}

\begin{proof}
(i) $\Rightarrow$ (ii):
Let $R=I \times \Lambda$. Then every element in the direct product
$S_I\times S_{\Lambda}$ is an automorphism of $R$. 

(ii) $\Rightarrow$ (i):
Let $e \in R$ be an idempotent. For every $a\in R$ there exists an
automorphism $\pi$ of $R$ that sends $e$ to $a$. Hence $R$ consists
entirely of idempotents. As a band $R$ is a semilattice of rectangular
bands. Clearly, the set of rectangular bands, that is the set of
$\mc{D}$-classes of $R$, is preserved by every automorphism. Hence
every automorphism induces an automorphism of the semilattice. The
automorphism group of a finite semilattice is transitive if and only
if the semilattice is trivial. Therefore $R$ consists of a single
rectangular band.\qed
\end{proof}

We conclude noting that the previous proposition implies that there
are arbitrarily high orders, all prime numbers, for which the full
symmetric group is the only transitive automorphism group for a
semigroup of the given order.

\begin{acknowledgements}
We thank Robert Gray, James Mitchell and Steve Linton for helpful
discussions, and James Mitchell and Csaba Schneider for comments 
on earlier versions of the paper.

The first author acknowledges the financial support from the doctoral
program of the University of St Andrews and from the project
PTDC/MAT/101993/2008 of Centro de \'Algebra da Universidade de Lisboa,
financed by FCT and FEDER.

The second author acknowledges the financial support by United Kingdom
Engineering and Physical Sciences Research Council (EPSRC) grant
EP/H004092/1.
\end{acknowledgements}

\bibliographystyle{spmpsci}      %

\end{document}